\documentclass{amsart}[12pt]

\let\oldmarginpar\marginpar
\renewcommand\marginpar[1]{\-\oldmarginpar[\raggedleft\footnotesize #1]%
{\raggedright\footnotesize #1}}

\usepackage{xypic}
\usepackage{eucal}

\usepackage{amsmath}
\usepackage{amstext}
\usepackage{amssymb}
\usepackage{amsthm}
\usepackage{amscd}

\usepackage{lscape}
\usepackage{longtable}

\usepackage{epsfig}
\input txdtools
\let\et=\etexdraw
\def\etexdraw{\drawbb\et}

\RequirePackage[dvipsnames,usenames]{color}

\theoremstyle{plain}
\newtheorem{thm}{Theorem}[section]
\newtheorem{thm*}{Theorem}
\newtheorem{lem}[thm]{Lemma}

\newtheorem{prop*}[thm*]{Proposition}

\theoremstyle{definition}
\newtheorem{defi}[thm]{Definition}

\theoremstyle{remark}

\newtheorem{rmk}[thm]{Remark}

\begin{document}

\title
[Two interesting examples of $\mathcal{D}$-modules]
{Two interesting examples of $\mathcal{D}$-modules in characteristic $p>0$}

\author{Mordechai Katzman}
\address{Department of Pure Mathematics,
University of Sheffield, Hicks Building, Sheffield S3 7RH, United Kingdom}
\email{M.Katzman@sheffield.ac.uk}

\author{Gennady Lyubeznik}
\address{School of Mathematics, University of Minnesota, Minneapolis, MN, 55455, USA}
\email{gennady@math.umn.edu}

\author{Wenliang Zhang}
\address{Department of Mathematics, University of Michigan, Ann Arbor, MI, 48109, USA}
\email{wlzhang@umich.edu}
\date{}


\begin{abstract}
We provide two examples of $\mathcal{D}$-modules in prime characteristic $p$ which answer two open problems in \cite{Lyubeznik} in the negative.
\end{abstract}

\maketitle


\section{Introduction}

Let $K$ be a field, let $R=K[x_1,\dots,x_n]$ be the ring of polynomials in $x_1,\dots,x_n$ over $K$ and let $\mathcal D$ be the ring of $K$-linear differential operators over $K$. In a remarkable paper \cite{Bavula} V. Bavula gave a characteristic-free definition of holonomic $\mathcal D$-modules. 
In characteristic zero his definition coincides with the usual one.  He proved, among other things, that his holonomic modules have one of the most important properties known from the characteristic zero case, namely, their length in the category of $\mathcal D$-modules is finite.

Using Bavula's ideas Lyubeznik \cite{Lyubeznik} gave a characteristic-free proof that $R_f$, for every non-zero element $f\in R$, is holonomic. This provided the first characteristic-free proof of the well-known fact that $R_f$ has finite length in the category of $\mathcal D$-modules.

In view of these developments it is interesting to see whether in characteristic $p>0$ holonomic modules, as defined by Bavula, have other properties known from the characteristic zero case.

Bavula proved that a submodule and a quotient module of a holonomic $\mathcal D$-module are holonomic. But in characteristic 0 it is also true that an extension of two holonomic modules is holonomic. Does this property hold in characteristic $p>0$ as well?

Let $\mathcal F_0\subset \mathcal F_1\subset\dots$ be the Bernstein filtration on $\mathcal D$, let $M$ be a holonomic $\mathcal D$-module generated by a finite set of elements $m_1\dots,m_s\in M$ and let $M_0\subset M_1\subset\dots$ be the filtration on $M$ defined by $M_i=\mathcal F_im_1+\cdots+\mathcal F_im_s$. In characteristic 0 it is well-known that dim$_kM_i$, for $i>>0$, is a polynomial in $i$ of degree $n$; in particular, lim$_{i\to\infty}\frac{{\rm dim}M_i}{i^n}$ exists and is finite. Does this property hold in characteristic $p>0$ as well?

These two questions were raised in the last section of \cite{Lyubeznik}. In this paper we give counter-examples to both of them. In Section 3 we produce a non-holonomic extension of two holonomic modules in characteristic $p>0$ and in Section 4 we produce a holonomic $\mathcal D$-module in characteristic $p>0$ such that the function dim$_kM_i$ is very far from a polynomial and in particular, lim$_{i\to\infty}\frac{{\rm dim}M_i}{i^n}$ does not exist.

\section{Preliminaries}
As explained in \cite[Section 2]{Lyubeznik}, a $K$-basis of $\mathcal D$ is the set of products $x_1^{i_1}\cdots x_n^{i_n}D_{t_1,1}\cdots D_{t_n,n}$ where $D_{t,i}=\frac{1}{t!}\frac{\partial^t}{\partial x_i^t}: R\to R$ is the $K[x_1, ...,x_{i-1} ,x_{i+1},...,x_n]$-linear map that sends $x_i^v$ to $\binom{v}{t}x_i^{v-t}$ ($D_{0,i}$ is the identity map) and $i_1,...,i_n,t_1,...,t_n$ range over all the 2n-tuples of non-negative integers. The Bernstein filtration $\mathcal F_0\subset \mathcal F_1\subset\dots$ on $\mathcal D$ is defined by setting $\mathcal F_s$ to be the $K$-linear span of the products $x_1^{i_1}\cdots x_n^{i_n}D_{t_1,1}\cdots D_{t_n,n}$ with $i_1+\cdots+i_n+t_1+\cdots+t_n\leq s$. It is not hard to see that $\mathcal F_i\cdot \mathcal F_j\subset \mathcal F_{i+j}$.

By a $\mathcal D$-module we always mean a left $\mathcal D$-module. By a $K$-filtration on a $\mathcal D$-module $M$ we mean an ascending chain of $K$-vector spaces $M_0\subset M_1\subset \dots$ such that $\cup_iM_i=M$ and $\mathcal F_iM_j\subset M_{i+j}$. Bavula's definition of a holonomic $\mathcal D$-module \cite[p. 185]{Bavula}, as simplified by Lyubeznik \cite[3.4]{Lyubeznik} is the following.

\begin{defi} A $\mathcal D$-module $M$ is holonomic if it has a $K$-filtration $M_0\subset M_1\subset\dots$ such that ${\rm dim}_kM_i\leq Ci^n$ for all $i$, where $C$ is a constant independent of $i$.
\end{defi}

It is straightforward to see that every submodule and every quotient module of a holonomic module are holonomic. Some other properties are that the length of a holonomic module $M$ in the category of $D$-modules is at most $n!C$ \cite[9.6]{Bavula} \cite[3.5]{Lyubeznik} (in particular, the length is finite) and $R_f$, for every $0\ne f\in R$, with its natural $\mathcal D$-module structure, is holonomic \cite[3.6]{Lyubeznik}.

For the rest of this paper $K$ denotes a perfect field of prime characteristic $p$. Let $\mathcal D_s$ be the (left) $R$-submodule of $\mathcal D$ generated by the products $D_{t_1,i_1}\cdots D_{t_n,i_n}$ such that $t_i<p^s$ for every $i$. It is not hard to see that $\mathcal D_s$ is a ring which (viewing $\mathcal D$ as a subring of Hom$_k(R,R)$) is nothing but Hom$_{R^{p^s}}(R,R)$. In particular, $\mathcal D=\cup_s\mathcal D_s$.

Our method of specifying a $\mathcal{D}$-module is as follows:
we start with a sequence of $\{ M^{(i)} \}_{i\geq 0}$ where each $M^{(i)}$ is a $R^{p^i}$-module
and $R^{p^{i+1}}$-linear maps
$\Theta_i : M^{(i+1)} \rightarrow M^{(i)}$ such that the $R^{p^i}$-module maps 
$\psi_i:R^{p^i}\otimes_{R^{p^{i+1}}}M^{(i+1)} \xrightarrow[]{r\otimes m\mapsto r\Theta(m)} M^{(i)}$ are bijective. This induces an $R$-module isomorphism $\phi_i:R\otimes_{R^{p^{i+1}}}M^{(i+1)}=R\otimes_{R^{p^i}}(R^{p^i}\otimes_{R^{p^{i+1}}}M^{(i+1)})\xrightarrow[]{{\rm id}\otimes\psi_i}R\otimes_{R^{p^i}}M^{(i)}$. Clearly, the compositions $\varphi_i=\phi_0\circ\phi_1\circ\cdots\circ\phi_i:R\otimes_{R^{p^{i+1}}}M^{(i+1)}\to M^{(0)}$ are $R$-module isomorphisms. The natural action of $\mathcal D_s$ on $R$ makes $R\otimes_{R^{p^s}}M^{(s)}$ a $\mathcal D_s$-module. This induces a structure of $\mathcal D_s$-module on $M^{(0)}$ via the isomorphism $\varphi_i$. It is not hard to check that if $s<s'$, then the $\mathcal D_{s'}$- and the $\mathcal D_s$-module structures thus defined are compatible with the natural inclusion $\mathcal D_s\subset \mathcal D_{s'}$, i.e. the $\mathcal D_s$-module structure is obtained from the $\mathcal D_{s'}$-module structure via the restriction of scalars. Since $\mathcal D=\cup_s\mathcal D_s$, this gives $M^{(0)}$ a structure of $\mathcal D$-module.

Both examples in this paper are special cases of the construction described in section 1 of \cite{Bogvad}.
Each $M^{(i)}$ is a free $R^{p^i}$-module with free generators $s_1^{(i)}$ and $s_2^{(i)}$,
$\Theta_i(s_1^{(i+1)})=s_1^{(i)}$ and
$\Theta_i(s_2^{(i+1)})=g_i s_1^{(i)} + s_2^{(i)}$
where for all $i\geq 0$, $g_i$ is an element of $Rx^{p^i}$. Since the elements $\Theta_i(s_1^{(i+1)})$ and
$\Theta_i(s_2^{(i+1)})$ generate $M^{(i)}$ as $R^{p^i}$-module, the associated map $\psi_i$ (defined in the preceding paragraph) is surjective. Since $\psi_i$ is a map between two free $R^{p^i}$-modules of rank two, it is bijective.
If we write $\sigma_n=-\sum^n_{r=0}g_r$
the resulting $\mathcal{D}$-module structure on $M$ is given by
\[\partial_{p^n}(f_1,f_2)=(\partial_{p^n}f_1+(\partial_{p^n}\sigma_n)f_2,\partial_{p^n}f_2),\]
for all $n\geq 0$.

Note that we have a short exact sequence of $\mathcal{D}$-modules
\[ 0\to R\xrightarrow{\psi}M\xrightarrow{\phi}R\to 0, \]
where $\psi(f)=(f,0)$ and $\phi(f_1,f_2)=f_2$.
Even though this exact sequence splits in the category of $R$-modules, it does not necessarily split in the category of $\mathcal{D}$-modules.


Our examples below result from a judicious choice of the sequences $\{ g_i \}_{i\geq 0}.$

\section{An example of a non-holonomic extension of holonomic modules}\label{Section: The example}

The main result  in this section is Theorem \ref{Theorem: Holonomicity does not extend}
which answers \cite[Question 1 in \S 4]{Lyubeznik} in the negative.
We do so by analyzing the $\mathcal{D}$-module obtained by setting   $g_r=x^{p^r+p^{2r}}$ for all $r\geq 0$ in
the construction of $\mathcal{D}$-modules described in the introduction.

We start with the following calculation to which we will refer repeatedly.
\begin{lem}\label{Lemma: computation of derivatives}
For any integers $0\leq \alpha \leq \beta$ and $K\geq 0$ we have
$$
\partial_{p^k} \left( x^{p^\alpha} x^{p^\beta} \right)
=
\left\{
\begin{array}{ll}
x^{p^{\alpha}} +  x^{p^{\beta}} & \text{if } k=\alpha=\beta,\\
 x^{p^{\beta}} & \text{if } k=\alpha<\beta,\\
x^{p^{\alpha}} & \text{if } \alpha<\beta=k,\\
x^{p^{\alpha}} & \text{if } p=2, \alpha=\beta=k-1,\\
0 & \text{otherwise.}
\end{array}
\right.
$$
\end{lem}
\begin{proof}
We first note that
$$
\partial_{j} \left( x^{p^\alpha} \right)
=
\left\{
\begin{array}{ll}
x^{p^\alpha} & \text{if } j=0,\\
1, & \text{if } j=p^\alpha,\\
0, & \text{otherwise.}
\end{array}
\right.
$$

Recall that $\partial_j$ is $K[x^{p^\gamma}]$-linear whenever $j<p^\gamma$ so if
$K<\alpha$ we have
$$
\partial_{p^k} \left( x^{p^{\alpha}} x^{p^{\beta}} \right)= x^{p^{\alpha}} \partial_{p^k} x^{p^{\beta}}= x^{p^{\alpha}} x^{p^{\beta}} \partial_{p^k} 1=0.
$$
If $K=\alpha$ we use  \cite[Proposition 2.1]{Lyubeznik} to compute
$$
\partial_{p^k}  \left( x^{p^{\alpha}} x^{p^{\beta}} \right) = \sum_{j=0}^{p^k} \partial_j  x^{p^{\alpha}} \partial_{p^k-j} x^{p^{\beta}}=
\partial_0  x^{p^{\alpha}} \partial_{p^k} x^{p^{\beta}} +  \partial_{p^k}  x^{p^{\alpha}} \partial_{0} x^{p^{\beta}}=
\left\{
\begin{array}{ll}
x^{p^{\alpha}} +  x^{p^{\beta}},  & \text{if } k=\alpha=\beta,\\
 x^{p^{\beta}},  & \text{if } k=\alpha<\beta.
\end{array}
\right.
$$

If $\alpha<\beta=k$ we compute
$$
\partial_{p^k}  \left( x^{p^{\alpha}} x^{p^{\beta}} \right) = \sum_{j=0}^{p^k} \partial_j  x^{p^{\alpha}} \partial_{p^k-j} x^{p^{\beta}}=
\partial_0  x^{p^{\alpha}} \partial_{p^k} x^{p^{\beta}} +  \partial_{p^\alpha}  x^{p^{\alpha}} \partial_{p^k-p^\alpha} x^{p^{\beta}}=
 x^{p^{\alpha}} .$$

If $\alpha\leq \beta<k$ we compute
\begin{eqnarray*}
\partial_{p^k}  \left( x^{p^{\alpha}} x^{p^{\beta}} \right) &=& \sum_{j=0}^{p^k} \partial_j  x^{p^{\alpha}} \partial_{p^k-j} x^{p^{\beta}}\\
&=& \partial_0  x^{p^{\alpha}} \partial_{p^k} x^{p^{\beta}} +  \partial_{p^\alpha}  x^{p^{\alpha}} \partial_{p^k-p^\alpha} x^{p^{\beta}}\\
&=&\left\{
\begin{array}{ll}
x^{p^{\alpha}},   & \text{if } \alpha=\beta=k-1 \text{ and } p=2,\\
0,  & \text{if } \alpha<\beta \text{ or } \beta<k-1 \text{ or } p>2.
\end{array}
\right.
.
\end{eqnarray*}

\end{proof}

\begin{thm}\label{Theorem: Holonomicity does not extend}
The $\mathcal{D}$-module $M$ is not holonomic in the sense of \cite[Definition 3.4]{Lyubeznik}.
\end{thm}
\begin{proof}
Let $s_1=(1,0)$ and $s_2=(0,1)$, be the free generators of $M$.
Let $\left\{ \mathcal{F}_i \right\}_{i\geq 0}$ denote the Bernstein filtration of $\mathcal{D}$.
Let $\left\{ M_i  \right\}_{i\geq 0}$ be any $K$-filtration of $M$.
Our aim is to show that $\lim_{i\rightarrow\infty} \dim_K M_i/i=\infty$; we may shift the indices to ensure that  $s_1, s_2\in M_0$
and we henceforth assume that this holds.

Since  $M_i \supseteq \mathcal{F}_i M_0$, it is enough to show that
the function $d(i)= \dim_K \mathcal{F}_i M_0$ is such that
$\lim_{i\rightarrow\infty} d(i)/i = \infty$.

For any pair of integers $(j,k)$ with $j,k \geq 0$, we have $x^j \partial_{p^k} s_2=(x^j\partial_{p^k}\sigma_k) s_1$.
Now consider the set of elements \[E=\{ r_{j k}:=x^j\partial_{p^k}\sigma_k \,|\, j+p^k\leq p^i\}\subseteq  \mathcal{F}_{p^i} M_0\subseteq M_{p^i}.\]
Lemma \ref{Lemma: computation of derivatives} gives
$
\partial_{p^k} \left( x^{p^k} x^{p^{2k}} \right)= x^{p^{2k}}
$
hence $ \deg \partial_{p^k} g_k= p^{2k}$ and
$\deg r_{j k}=j + \deg \partial_{p^k} \sigma_k=j+\deg(\sigma_k)-p^k=j+p^{2k}$.
Hence for all different pairs $(j,k)$ and $(j',k')$ with $K,k'\geq i/2$, we have
\[\deg r_{j k}=j+p^{2k}\neq j'+p^{2k'}=\deg(r_{j' k'})\]
otherwise, $p^{2k}-p^{2k'}=j'-j$ which implies  $p^i|(j'-j)$, contradicting the fact that $j, j^\prime<p^i$.
We deduce that elements $r_{j k}$ of $E$ with $K\geq i/2$ and $j+p^k\leq p^i$ have distinct degrees and hence are linearly independent over $K$.
Let $\lceil i/2\rceil$ denote the least integer greater than or equal to $i/2$. For each $K\geq i/2$, there are $p^i-p^k$ many $r_{j k}$; therefore
\begin{align}
\dim_K \mathcal{F}_{p^i} M_0 & \geq \sum_{i\geq k\geq \lceil i/2\rceil }(p^i-p^k)\notag\\
& =(i-\lceil i/2\rceil+1) p^i-p^{\lceil i/2\rceil}\frac{p^{i-\lceil i/2\rceil+1}-1}{p-1}\notag
\end{align}
which implies that
\[
\lim_{i\rightarrow\infty} \frac{d(p^i)}{p^i}=
\lim_{i\rightarrow\infty} (i-\lceil i/2\rceil+1) + \frac{p^{i-\lceil i/2\rceil+1}-1}{p^{i-\lceil i/2\rceil}(p-1)}=
\infty.
\]

\end{proof}

\section{An example of a holonomic module whose multiplicity does not exist}

Let $M$ be as in the previous section with $g_r$ replaced by $g_r=x^{(p+1)p^r}$.
In this section we show that $\mathcal{D}s_2$, the $\mathcal{D}$-submodule of $M$ generated by $s_2$, is a holonomic $\mathcal{D}$-module for which
$\displaystyle\lim_{i\rightarrow\infty} \frac{\mathcal{F}_is_2}{i}$
does not exist. This gives a negative answer to \cite[Question 2 in \S 4]{Lyubeznik}


We start with the following calculation.

\begin{lem}
\label{partial-on-s2}
Let $K_1,\dots,k_t,e_1,\dots,e_t$ be nonnegative integers.

\begin{enumerate}

\item [(a)]
\[
\partial_{p^k}\sigma_k=
\left\{
\begin{array}{ll}
x^p, & \text{if } k=0,\\
x^{p^{k+1}}+x^{p^{k-1}}, &  \text{otherwise}.
\end{array}
\right.
\]

\item[(b)]
\[
(\partial_{p^{k_1}})^{e_1}\cdots (\partial_{p^{k_t}})^{e_t}s_2=
\left\{
\begin{array}{ll}
s_2, & \text{if } e_1=\cdots e_t=0\\
x^ps_1, & \text{if } t=1, e_1=1, k_1=0\\
(x^{p^{k_t+1}}+x^{p^{k_t-1}})s_1, & \text{if } t=1, e_1=1, k_t\geq 1\\
s_1, & \text{if } t=2, e_1=e_2=1, k_1=k_2+1 \text{ or } k_1=k_2-1\\
0, &  \text{otherwise}.
\end{array}
\right.
\]
\end{enumerate}
\end{lem}

\begin{proof}[Proof]
\begin{enumerate}

\item [(a)]
Lemma \ref{Lemma: computation of derivatives} implies that for $K\geq 0$,
$\partial_{p^k} g_r$ vanishes when $r+1<k$, that $\partial_{p^k} g_{k-1}=x^{p^{k-1}}$ and that  $\partial_{p^k} g_{k}=x^{p^{k-1}}=x^{p^{k+1}}$
hence
\[
\partial_{p^k}\sigma_k=
\left\{
\begin{array}{ll}
x^p, & \text{if } k=0,\\
x^{p^{k+1}}+x^{p^{k-1}}, &  \text{otherwise}.
\end{array}
\right.
\]

\item[(b)]
 This follows immediately from (a).
\end{enumerate}
\end{proof}

\begin{thm}
Let $S$ denote $\mathcal{D}s_2$ and $S_i$ denote $\mathcal{F}_is_2$.Then
\[\dim_k(S_i)=
\left\{
\begin{array}{ll}
2i+p^{e+1}-p^e+2, & \text{if } p^{e+1}-p^e+p^{e-1}\leq i<p^{e+1};\\3i-p^{e-1}+3,& \text{if } p^e\leq i< p^{e+1}-p^e+p^{e-1},
\end{array}
\right.
\]
where $e$ is the unique integer such that $p^e\leq i<p^{e+1}$.
Consequently, $S$ is holonomic and $\displaystyle\lim_{i\to \infty}\frac{\dim_k(S_i)}{i}$ does not exist.
\end{thm}
\begin{proof}[Proof]
Consider a general element $x^j(\partial_{p^{k_1}})^{e_1}\cdots (\partial_{p^{k_t}})^{e_t}s_2$ with $j+\sum^t_{i=0}e_ip^{k_i}\leq i$ in $S_i$.
Lemma \ref{partial-on-s2} shows that $x^j(\partial_{p^{k_1}})^{e_1}\cdots (\partial_{p^{k_t}})^{e_t}s_2$ equals
$$\left\{
\begin{array}{lll}
x^js_2, & \text{with } 0\leq j\leq i, & \text{if } e_1=\cdots e_t=0\\
x^jx^ps_1, & \text{with } 0\leq j\leq i-1, & \text{if } t=1, e_1=1, k_1=0\\
x^j(x^{p^{k_1+1}}+x^{p^{k_1-1}})s_1, & \text{with }0\leq j\leq i-p^{k_1} , & \text{if } t=1, e_1=1, k_1\geq 1 \\
x^js_1, & \text{with } 0\leq j\leq i-p^{k_2}-p^{k_2+1}, & \text{if } t=2, e_1=e_2=1, k_1=k_2+1\\ x^js_1, & \text{with } 0\leq j\leq i-p^{k_2}-p^{k_2-1}, & \text{if } t=2, e_1=e_2=1, k_1=k_2-1\\
0, &  \text{otherwise}. &
\end{array}
\right.
$$

From this we can see that, among
$\{x^j(\partial_{p^{k_1}})^{e_1}\cdots (\partial_{p^{k_t}})^{e_t}s_2|j+\sum^t_{i=0}e_ip^{k_i}\leq i\}$,
there are three types of elements:
elements obtained when $t=0$ (i.e., $e_1=\cdots e_t=0$), elements obtained when $t=1$, and elements obtained when $t=2$.
Let $V_1$ be the $K$-span of all elements of the first type, i.e.,
\[V_1=K\langle x^js_2|0\leq j\leq i\rangle .\]
Let $V_2$ be the $K$-span of all elements of the second type, i.e.,
\[V_2=K\langle x^jx^ps_1|0\leq j\leq i-1\rangle + K\langle x^j(x^{p^{k+1}}+x^{p^{k-1}})|0\leq j+p^k\leq i, k\geq 1\rangle.\]
Let $V_3$ be the $K$-span of all elements of the third type, i.e.,
\[V_3=K\langle x^js_1|0\leq j\leq i-1-p\rangle.\]

It should be pointed out that all elements of the first type are in the copy of $R$ generated by $s_2$ and all elements of the other two types are in the copy of $R$
generated by $s_1$; and hence
$\dim_K(\mathcal{F}_is_2)=\dim_K(V_1)+\dim_K(V_2+V_3)$ since $s_1$ and $s_2$ are linearly independent.
It is clear that $V_1$ is a $(i+1)$-dimensional $K$-vector space, it remains to calculate $\dim_K(V_2+V_3)$.

To calculate the dimension of $V_2+V_3$, we break $V_2$ into pieces as follows:
\begin{eqnarray*}
V_{2,0}&=&K\langle x^jx^ps_1|0\leq j\leq i-1\rangle \\
V_{2,1}&=&K\langle x^j(x^{p^2}+x)s_1|0\leq j\leq i-p\rangle \\
& \vdots & \\
V_{2,e-1}&=&K\langle x^j(x^{p^e}+x^{p^{e-2}})s_1|0\leq j\leq i-p^{e-1}\rangle \\
V_{2,e}&=&K\langle x^j(x^{p^{e+1}}+x^{p^{e-1}})s_1|0\leq j\leq i-p^{e}\rangle
\end{eqnarray*}
If $i\geq 2p+1$, $V_{2,0}+V_3=K\langle x^js_1|0\leq j\leq i-1+p\rangle$.
Since $V_{2,0}+V_3$ contains $x^jxs_1$ with $0\leq j\leq i-p$, we see that $V_{2,1}+V_{2,0}+V_3$ contains $x^jx^{p^2}s_1$ with $0\leq j\leq i-p$.
Consequently
\[V_{2,1}+V_{2,0}+V_3=K\langle x^js_1|0\leq j\leq i-p+p^2\rangle.\]
Similarly, we have
\[V_{2,e-1}+\cdots+V_{2,0}+V_3=K\langle x^js_1|0\leq j\leq i-p^{e-1}+p^e\rangle.\]

It remains to analyze $V_{2,e}=K\langle x^j(x^{p^{e+1}}+x^{p^{e-1}})s_1|0\leq j\leq i-p^{e}\rangle$. There are two cases:
\begin{enumerate}
\item[\underline{Case 1:}]
$i\geq p^{e+1}-p^e+p^{e-1}$, i.e. $[p^{e},i-p^{e-1}+p^e]\cap[p^{e+1},i-p^e+p^{e+1}]\neq \emptyset$.
In this case, similar to the consideration of $V_{2,e-1}+\cdots+V_{2,0}+V_3$,
 we can see that $V_{2,e}+\cdots+V_{2,0}+V_3$ consists of all polynomials of degree less than or equal to $i-p^e+p^{e+1}$.
Therefore,
\[\dim_K(V_2+V_3)=i-p^e+p^{e+1}+1,\]
and hence
\[\dim_K(\mathcal{F}_is_2)=\dim_K(V_1)+\dim_K(V_2+V_3)=2i-p^e+p^{e+1}+2.\]

\item[\underline{Case 2:}]
$i< p^{e+1}-p^e+p^{e-1}$, i.e. $[p^{e},i-p^{e-1}+p^e]\cap[p^{e+1},i-p^e+p^{e+1}]= \emptyset$.
In this case the degrees of the basis elements of  $V_{2,e}$ (which are all distinct) exceed the degree of any element in
$V_{2,e-1}+\cdots+V_{2,0}+V_3$ thus
$\dim_K V_{2,e} + V_{2,e-1}+\cdots+V_{2,0}+V_3= \dim_K V_{2,e} + \dim_K V_{2,e-1}+\cdots+V_{2,0}+V_3=(i-p^{e-1}+p^e+1)+(i-p^e+1)$ and
\[\dim_K(\mathcal{F}_is_2)=(i+1)+(i-p^{e-1}+p^e+1)+(i-p^e+1)=3i-p^{e-1}+3.\]
\end{enumerate}
We note that for all $i$, we have $\dim_K(S_i)\leq 4i$, therefore $S=\mathcal{D}s_2$ is holonomic. \par

For all $i$ of the form $p^e$, we have
\[\lim_{e\to \infty}\frac{\dim_K(S_{p^e})}{p^e}=\lim_{e\to \infty}\frac{3p^e-p^{e-1}+3}{p^e}=3-\frac{1}{p};\]
but for all $i$ of the form $p^{e+1}-p^e$, we have
 \[\lim_{e\to \infty}\frac{\dim_K(S_{p^{e+1}-p^e})}{p^{e+1}-p^e}=\lim_{e\to \infty}\frac{3(p^{e+1}-p^e)-p^{e-1}+3}{p^{e+1}-p^e}=3-\frac{1}{p^2-p}.\]
Therefore, $\displaystyle\lim_{i\to \infty}\frac{\dim_K(S_i)}{i}$ does not exist.
\end{proof}

\begin{rmk}\hfil
\begin{enumerate}
\item The proof of the previous theorem shows that $M=\mathcal{D}s_2$ is a \emph{cyclic} $\mathcal{D}$-module .
\item A reasonable theory of holonomic modules should include the polynomial ring itself as a holonomic module.
However our example in this section indicates that any such theory of holonomic modules can not have both the extension property and the existence of multiplicity
at the same time. If extensions of holonomic modules are holonomic, then our module $M=\mathcal{D}s_2$ in this section will be holonomic,
but as we have seen the multiplicity of $M$ does not exist.
\end{enumerate}
\end{rmk}

\subsection*{Acknowledgments.} The results in this paper were obtained while the first and third authors enjoyed the hospitality of the School of Mathematics at the University of Minnesota.

\end{document}